\documentclass[11pt]{amsart}
\usepackage{
 amsmath, 
 amsxtra, 
 amsthm, 
 amssymb, 
 etex, 
 mathrsfs, 
 mathtools, 
 tikz-cd, 
 bbm,
 xr,
 comment,
 enumitem}
\usepackage[toc]{appendix} 
\usepackage[all]{xy}
\usepackage{hyperref}
\usepackage{url}
\usepackage{tipa}
\usepackage{dsfont}
\usepackage{ textcomp }
\usepackage{stmaryrd} 
\usepackage{amssymb}
\usepackage{mathabx} 
\usepackage{amsmath} 
\usepackage{amscd}
\usepackage{amsbsy}
\usepackage{comment, enumerate}
\usepackage{color}
\usepackage{mathtools,caption}
\usepackage{tikz-cd}
\usepackage{longtable}
\usepackage[utf8]{inputenc}
\usepackage[OT2,T1]{fontenc}
\usepackage{changepage}
\usepackage{commath,enumerate}

\DeclareMathOperator{\Gal}{Gal}
\DeclareMathOperator{\rank}{rank}

\newcommand{\tors}{\mathrm{tors}}

\DeclareMathOperator{\Sel}{Sel}

\newtheorem{theorem}{Theorem}[section]
\newtheorem*{theorem*}{Theorem}
\newtheorem{lemma}[theorem]{Lemma}

\newtheorem{proposition}[theorem]{Proposition}
\newtheorem{corollary}[theorem]{Corollary}

\numberwithin{equation}{section}
\newtheorem{lthm}{Theorem} 

\theoremstyle{remark}
\newtheorem{remark}[theorem]{Remark}

\newcommand\EatDot[1]{}

\newcommand{\cH}{{\mathcal{H}}}

\newcommand{\loc}{\mathrm{loc}}

\newcommand{\cS}{{\mathcal{S}}}
\newcommand{\cX}{\mathcal{X}}
\newcommand{\QQ}{\mathbb{Q}}
\newcommand{\ZZ}{\mathbb{Z}}
\newcommand{\Qp}{\mathbb{Q}_p}
\newcommand{\Zp}{\mathbb{Z}_p}
\newcommand{\Srel}{\cS_\mathrm{rel}}
\newcommand{\Xstr}{\cX_\mathrm{str}}

\definecolor{Green}{rgb}{0.0, 0.5, 0.0}

\newcommand{\cO}{\mathcal{O}}
\newcommand{\cC}{\mathcal{C}}

\newcommand{\fm}{\mathfrak{m}}

\newcommand{\Char}{\mathrm{Char}}

\newcommand{\Xrel}{\cX_\mathrm{rel}}

\newcommand{\Cloc}{\cC_\mathrm{loc}}
\newcommand{\Sstr}{\cS_\mathrm{str}}
  \DeclareFontFamily{U}{wncy}{}
  \DeclareFontShape{U}{wncy}{m}{n}{<->wncyr10}{}
  \DeclareSymbolFont{mcy}{U}{wncy}{m}{n}
  \DeclareMathSymbol{\sha}{\mathord}{mcy}{"58}
  \DeclareMathSymbol{\zhe}{\mathord}{mcy}{"11}
\title[Characteristic elements of anticyc. Selmer groups of CM curves]{A remark on the characteristic elements of anticyclotomic Selmer groups of elliptic curves with complex multiplication at supersingular primes}

\makeatletter
\let\@wraptoccontribs\wraptoccontribs
\makeatother
\usepackage{tikz,xcolor,hyperref}

\definecolor{lime}{HTML}{A6CE39}
\DeclareRobustCommand{\orcidicon}{%
	\begin{tikzpicture}
	\draw[lime, fill=lime] (0,0) 
	circle [radius=0.16] 
	node[white] {{\fontfamily{qag}\selectfont \tiny ID}};
	\draw[white, fill=white] (-0.0625,0.095) 
	circle [radius=0.007];
	\end{tikzpicture}
	\hspace{-2mm}
}

\foreach \x in {A, ..., Z}{%
	\expandafter\xdef\csname orcid\x\endcsname{\noexpand\href{https://orcid.org/\csname orcidauthor\x\endcsname}{\noexpand\orcidicon}}
}

\author[A.~Lei]{Antonio Lei\orcidA{}}
\address[Lei]{Department of Mathematics and Statistics\\University of Ottawa\\
150 Louis-Pasteur Pvt\\
Ottawa, ON\\
Canada K1N 6N5}
\email{antonio.lei@uottawa.ca}


\begin{document}
\begin{abstract}
Let $p\ge5$ be a prime number. Let $E/\mathbb{Q}$ be an elliptic curve with complex multiplication by an imaginary quadratic field $K$ such that $p$ is inert in $K$ and that $E$ has good reduction at $p$. Let $K_\infty$ be the anticyclotomic $\mathbb{Z}_p$-extension of $K$. Agboola--Howard defined Kobayashi-type signed Selmer groups of $E$ over $K_\infty$ and showed that exactly one of them is cotorsion over the corresponding Iwasawa algebra. In this short note, we discuss a link between the characteristic ideals of the cotorsion signed Selmer group and the fine Selmer group building on a recent breakthrough of Burungale--Kobayashi--Ota on the structure of local points. 
\end{abstract}

\maketitle

\section{Introduction}
\label{S: Intro}
Throughout this note, $p\ge5$ is a fixed prime number. Let $E/\QQ$ be an elliptic curve with complex multiplication by an imaginary quadratic field $K$ where $p$ is inert. Furthermore, we assume that $p$ does not divide the conductor of $E$. In particular, $E$ has good supersingular reduction at $p$ and $a_p(E)=0$.
We write $\Gamma=\Gal(K_\infty/K)$ and $\Lambda=\cO[[\Gamma]]$, where $\cO$ is the ring of integers of the unramified quadratic extension of $\Qp$.

Similar to the cyclotomic theory, anticyclotomic Iwasawa Theory at supersingular primes is more intricate than the ordinary counterpart. Inspired by earlier works of Rubin \cite{rubin87} and Kobayashi \cite{kobayashi03}, Agboola--Howard \cite{agboolahowardsupersingular} defined the anticyclotomic plus and minus Selmer groups attached to $E$. We are interested in understanding links between these plus and minus Selmer groups and the fine  Selmer group (also called strict Selmer group) of $E$ over $K_\infty$ generalizing recent results on their cyclotomic counterparts in \cite{LeiSujatha2,LL21}.

Let $\cX_\pm$  be the Pontryagin duals of the $p$-primary plus and minus (discrete) Selmer groups $\Sel_{\pm}(K_\infty,E[p^\infty])$ defined in \cite{agboolahowardsupersingular}.
Let $\epsilon\in\{+,-\}$ be the sign of the functional equation of $L(E/\QQ,s)$. Agboola--Howard  showed in \cite[Theorem~3.6]{agboolahowardsupersingular} that $\cX_\epsilon$ is of rank one over $\Lambda$, whereas $\cX_{-\epsilon}$ is a torsion $\Lambda$-module. The main goal of this note is to study relations between $\cX_{-\epsilon}$ and the Pontryagin dual of the fine Selmer group of $E$ over $K_\infty$, denoted by $\Xstr$ in the present note. Let $L_{-\epsilon}$ (respectively $L_0$) be a characteristic element of $\cX_{-\epsilon}$ (respectively $\Xstr$). Building on the validity of Rubin's conjecture on local plus and minus points due to Burungale--Kobayashi--Ota, we prove:
\begin{lthm}[{Theorem~\ref{thm:main}}]\label{thmA}
As elements of $\Lambda$, we have the divisibilities
\[
\mathrm{lcm}(L_0,L_0^\iota)\large| L_{-\epsilon}\large| L_0L_0^\iota.
\]
\end{lthm}
Here, $\iota$ denotes the involution on $\Lambda$ sending a group-like element $\sigma\in\Gamma$ to $\sigma^{-1}$ and $\mathrm{lcm}$ signifies the "least common multiple", which is well-defined up to a unit in $\Lambda$ (since $\Lambda$ is a unique factorization domain).

The statement of Theorem~\ref{thmA} is somewhat simpler and more precise than the cyclotomic analogues in \cite{LeiSujatha2,LL21}. The fact that only one of the two signed Selmer groups is cotorsion over $\Lambda$ allows us to describe  more precisely the $\Lambda$-structure of one of the modules we study (see Proposition~\ref{prop:H1/Cloc}). Such a result does not seem to be available in the cyclotomic case and is one of the key diverging points between the current and previous works.

\subsection*{Acknowledgement}
We thank Ashay Burungale for interesting discussions related to \cite{BKO} during the preparation of this note. We also thank Meng Fai Lim and Katharina Müller for helpful exchanges. The author is grateful to the anonymous referee for helpful comments and suggestions, which led to many improvements in this note. The author's research is supported by the NSERC Discovery Grants Program RGPIN-2020-04259 and RGPAS-2020-00096.

\section{Plus and minus Selmer groups}

For an integer $n\ge0$, $K_n$ denotes the $n$-th layer of the anticyclotomic tower $K_\infty/K$. By an abuse of notation, the unique prime of $K_n$ above $p$ will be denoted by $p$. We write $\Psi_n$ for the completion of $K_n$ at $p$. The maximal ideal of the ring of integers of $\Psi_n$ is denoted by $\fm_n$.
Let $G_n=\Gal(K_n/K)\cong \Gal(\Psi_n/\Psi_0)$.
 We define $\Xi_n^+$ (respectively $\Xi_n^-$) the set of characters of $G_n$ whose exact orders are even (respectively odd) powers of $p$. Let $\hat{E}$ denote the formal group of $E$ over $\Psi_0$. For $n\ge0$, we define the plus and minus subgroups
 \[
 \hat E(\Psi_n)^\pm=\left\{x\in\hat E(\fm_n):\sum_{\sigma\in G_n}\chi(\sigma)\log_{\hat 
 E}\left(x^\sigma\right)=0,\ \forall \chi\in \Xi_n^\mp \right\},
 \]
 where $\log_{\hat E}$ denote the formal logarithm of $\hat E$.
 
The Kummer map induces an injection $$\hat{E}(\fm_n)^\pm\otimes\ \Qp/\Zp\hookrightarrow H^1(\Psi_n,E[p^\infty]).$$ We write $H^1_\pm(\Psi_n,E[p^\infty])$ for its image.
Let $T$ be the $p$-adic Tate module of $E$.  Let $H^1_\pm(\Psi_n,T)$ be the exact annihilator of $H^1_\pm(\Psi_n,E[p^\infty])$ under the Tate pairing
\[
H^1(\Psi_n,E[p^\infty])\times H^1(\Psi_n,T)\rightarrow\Qp/\Zp.
\]
We write $\cH^1_\pm$ (respectively $\cH^1)$ for the inverse limits $\varprojlim_n H^1_\pm(\Psi_n,T)$ (respectively $\varprojlim_n H^1(\Psi_n,T)$), where the connecting maps are corestrictions. The following proposition describes the structure of these $\Lambda$-modules.

\begin{proposition}
\label{prop:structure-H}
       \item[(i)] $\cH^1=\cH^1_+\oplus\cH^1_-$.
\item[(ii)] $\cH^1_\pm$ are free of rank one over $\Lambda$.
\end{proposition}
\begin{proof}
  This follows from \cite[Proposition~8.1]{rubin87} and  \cite[Theorem~5.8]{BKO}.
\end{proof}

We now define various Selmer groups that will be utilized and studied in this note. We follow closely the notation from \cite{agboolahowardsupersingular}.
Let $F=K_n$ for some $n\ge0$ and $v$ a place of $F$. We define $H^1_f(F_v,E[p^\infty])$ to be the image of $E(F_v)\otimes\Qp/\Zp$ inside $H^1_f(F_v,E[p^\infty])$ under the Kummer map and write $H^1_f(F_v,T)\subset H^1(F_v,T)$ for the orthogonal complement of $H^1_f(F_v,E[p^\infty])$ under the local Tate pairing. For $M=E[p^\infty]$ or $M=T$,  we define the \textit{relaxed} Selmer group 
\[
\Sel_\mathrm{rel}(F,M)=\ker\left(H^1(F,M)\rightarrow \prod_{v\nmid p} \frac{H^1(F_v,M)}{H^1_f(F_v,M)}\right),
\]
where the product runs over all places of $F$ not dividing $p$. We also define the \textit{strict} Selmer group by
\[
\Sel_\mathrm{str}(F,W)=\ker\left(\Sel_\mathrm{rel}(F,W)\rightarrow  H^1(\Psi_n,W)\right),
\]
and the plus and minus Selmer groups by
\[
\Sel_\pm(F,W)=\ker\left(\Sel_\mathrm{rel}(F,W)\rightarrow \frac{ H^1(\Psi_n,W)}{H^1_\pm(\Psi_n,W)}\right).
\]

For $\star\in\{\mathrm{rel},\mathrm{str},+,-\}$, we define
\[
\Sel_\star(K_\infty, E[p^\infty])=\varinjlim_n \Sel_\star(K_n,E[p^\infty]),\quad \text{and}\quad \cS_\star=\varprojlim_n \Sel_\star(K_n,T).
\]
Finally, the Pontryagin dual of $\Sel_\star(K_\infty, E[p^\infty])$ will be denoted by $\cX_\star$.

\section{Relating signed Selmer groups to fine Selmer groups}
Let $\cC$ be the $\Lambda$-module generated by the elliptic units in $\Srel$ as defined in \cite[\S3]{agboolahowardsupersingular}. We write $\Cloc$ for the image of $\cC$ in $\cH^1$ under the localization map. We review certain properties of $\cC$ as well as $\Lambda$-structures of  various Selmer groups that were proved in \cite{agboolahowardsupersingular}.
\begin{proposition}\label{prop:AH}
\item[(i)]$\Cloc$ lies in $\cH^1_\epsilon$ (in particular, $\cC\subseteq\cS_\epsilon$).
\item[(ii)]As $\Lambda$-modules, $\Srel$ is torsion-free of rank one, $\Xstr$ is torsion and $\Xrel$ has rank one.
\item[(iii)]$\Char_\Lambda \Xstr=\Char_\Lambda \Srel/\cC$.
\item[(iv)]$\rank_\Lambda\cX_\epsilon=\rank_\Lambda\cS_\epsilon=1$, $\rank_\Lambda\cX_{-\epsilon}=0$  and $\cS_{-\epsilon}=\cS_\mathrm{str}=0$.
\end{proposition}
\begin{proof}
This is \cite[Propositions~3.1 and 3.3, and Theorem 3.6]{agboolahowardsupersingular}.
\end{proof}

As in the introduction, we write $L_{-\epsilon}$ (respectively $L_0$) for a characteristic element of $\cX_{-\epsilon}$ (respectively $\Xstr$).
\begin{lemma}\label{lem:quotient}
    There is a pseudo-isomorphism of $\Lambda$-modules
    \[
    \Lambda/(L_{-\epsilon})\sim \cH^1_\epsilon /\cC_\loc.
    \]
\end{lemma}
\begin{proof}
     Proposition~\ref{prop:structure-H}(ii) says that $\cH_\epsilon$ is free of rank one over $\Lambda$. Therefore, there is a pseudo-isomorphism $\cH^1_\epsilon/\cC_\loc\rightarrow \Lambda/(F)$ for some $F\in\Lambda$. As discussed in \cite[end of \S4]{agboolahowardsupersingular}, the characteristic ideal of $\cH^1_\epsilon/\cC_\loc$ is generated by $L_{-\epsilon}$, from which the lemma follows.
\end{proof}

\begin{proposition}\label{prop:H1/Cloc}
There is a pseudo-isomorphism of $\Lambda$-modules
\[
\cH^1/\Cloc\sim \Lambda\oplus\Lambda/(L_{-\epsilon}).
\]
\end{proposition}
\begin{proof}
Propositions~\ref{prop:structure-H}(i) and \ref{prop:AH}(i) imply that there is a $\Lambda$-isomoprhism
\[
\cH^1/\Cloc\cong\cH^1_{-\epsilon} \bigoplus\cH^1_\epsilon/\Cloc .
\]
Hence, the result follows from Proposition~\ref{prop:structure-H}(ii) and Lemma~\ref{lem:quotient}.
\end{proof}

By \cite[proof of Theorem~3.6, (3.3)]{agboolahowardsupersingular}, there is an exact sequence:
\begin{equation}\label{eq:PT}
   0\rightarrow  \cS_\pm\rightarrow\Srel\rightarrow\cH^1_\mp\rightarrow\cX_\pm\rightarrow\Xstr\rightarrow0,
\end{equation}
where we have replaced $\cH^1$ in loc. cit. by the direct sum $\cH^1_+\bigoplus\cH^1_-$ (as given by Proposition~\ref{prop:structure-H}(i)).
Combining Proposition~\ref{prop:AH}(iv) and \eqref{eq:PT} gives the following exact sequence:
\begin{equation}
    0\rightarrow \Srel\rightarrow\cH^1_\epsilon\rightarrow\cX_{-\epsilon}\rightarrow\Xstr\rightarrow0.
\label{eq:PT-epsilon}
\end{equation}
It then follows that the localization map induces an injection $\Srel\hookrightarrow \cH^1$.
By an abuse of notation, we shall denote the image of $\Srel$ inside $\cH^1$ by the same notation. Proposition~\ref{prop:AH}(ii) tells us that the quotient $\cH^1/\Srel$ is of rank one over $\Lambda$. We can describe the torsion part of this module:

\begin{corollary}\label{cor:H1/rel}We have $\Char_\Lambda(\cH^1/\Srel)_\tors=(L_{-\epsilon}/L_0)$. 
\end{corollary}
\begin{proof}
Consider the short exact sequence:
\[
0\rightarrow \Srel/\cC\rightarrow \cH^1/\Cloc\rightarrow \cH^1/\Srel\rightarrow 0.
\]
Since the first term is $\Lambda$-torsion, with characteristic ideal generated by $L_0$ (thanks to Proposition~\ref{prop:AH}(iii)), 
\cite[Proposition~2.1]{HL2} tells us that
\[
\Char_\Lambda(\cH^1/\Srel)_\tors(L_0)=\Char_\Lambda(\cH^1/\Cloc)_\tors. 
\]Hence, the result follows from Proposition~\ref{prop:H1/Cloc}.
\end{proof}
We are now ready to prove Theorem~\ref{thmA} stated in the introduction.
\begin{theorem}\label{thm:main}
As elements of  $\Lambda$, we have the divisibilities
\[
\mathrm{lcm}(L_0,L_0^\iota)\large| L_{-\epsilon}\large| L_0L_0^\iota.
\]
\end{theorem}
\begin{proof}
Since $\Sstr=0$, we have  the Poitou--Tate exact sequence
\begin{equation}
0\rightarrow  \Srel\rightarrow\cH^1\rightarrow\Xrel\rightarrow\Xstr\rightarrow0,
\label{eq:PT2}    
\end{equation}
which gives the short exact sequence
\[
0\rightarrow\cH^1/\Srel\rightarrow\Xrel\rightarrow\Xstr\rightarrow0.
\]
This gives
\[
(\cH^1/\Srel)_\tors\hookrightarrow(\Xrel)_\tors\quad\text{and}\quad (\Xrel)_\tors\large /(\cH^1/\Srel)_\tors\hookrightarrow \Xstr.
\]
Since $E$ is supersingular at $p$, the classical $p$-primary Selmer group (called the \textit{true} Selmer group in \cite{agboolahowardsupersingular}) $\Sel(K_\infty,E[p^\infty])$ is equal to $\Sel_\mathrm{rel}(K_\infty,E[p^\infty])$ (see \cite[Remarque~3.3]{billot}). Thus, it follows from \cite[Corollary~2.5]{wingberg} that $\Char_\Lambda(\Xrel)_\tors=\Char_\Lambda (\Xstr)^\iota$. Combined with Corollary~\ref{cor:H1/rel}, we deduce
\[
L_{-\epsilon}/L_0|L_0^\iota, \quad L_0^\iota L_0/L_{-\epsilon}|L_0,
\]
or equivalently 
\[L_{-\epsilon}|L_0L_0^\iota,\quad L_0^\iota| L_{-\epsilon}.\]
Finally, the surjectivity of the last arrow in \eqref{eq:PT-epsilon} gives  $L_0|L_{-\epsilon}$. This concludes the proof of the theorem.
\end{proof}
We have the following immediate corollaries:
\begin{corollary}\label{cor:equal}
If $L_0$ is coprime to $L_0^\iota$, then $(L_{-\epsilon})=(L_0L_0^\iota)$.
\end{corollary}
\begin{corollary}\label{cor:square}
If $(L_0)=(L_0^\iota)$ as ideals in $\Lambda$, then $L_{-\epsilon}|L_0^2$.
\end{corollary}
\begin{corollary}
\label{cor:LS}
Let $F\in \Lambda$ be an irreducible element. Then $F\nmid L_{-\epsilon}$ if and only if $F\nmid L_0$ and $F^\iota\nmid L_0$.
\end{corollary}

\begin{corollary}\label{cor:mu}
The $\mu$-invariant of $\Xstr$ is zero if and only if the $\mu$-invariant of $\cX_{-\epsilon}$ is zero. 
\end{corollary}

\begin{corollary}
 $\Xstr$ is pseudo-null if and only if $\cX_{-\epsilon}$ is pseudo-null.
\end{corollary}
\begin{remark}
In light of \cite[Conjecture~A]{CoatesSujatha_fineSelmer} and \cite[Conjecture~B]{matar18}, it seems  plausible to expect  the $\mu$-invariant of $\Xstr$ to vanish in our current setting. Corollary~\ref{cor:mu} gives an equivalence of this expectation in terms of $\cX_{-\epsilon}$. 
\end{remark}
We conclude this note with the following speculative remark.
\begin{remark}
Over the cyclotomic $\Zp$-extension of $\QQ$, Greenberg predicted that the characteristic ideal of the dual fine Selmer group is generated by cyclotomic polynomials, whose exponents are related to the growth of the Mordell--Weil ranks of the elliptic curve inside the $\Zp$-extension (see \cite[Problem~0.7]{KP}). 

 Let 
\[
\eta_n=\begin{cases}\frac{\rank_\ZZ E(K_n)-\rank_\ZZ E(K_{n-1})}{p^{n-1}(p-1)}&n>0,\\
\rank_\ZZ E(K)&n=0.\end{cases}
\]
It is known that $\eta_n$ is always an even integer and that $\eta_n=2$ (resp. 0) for sufficiently large $n$ if  parity is the same as (resp. opposite of) $\epsilon$ (see for example \cite[P.6]{Gre_PCMS}). We set \[
e_n=\begin{cases}
\eta_n/2-1&\text{if the parity of $n$ is the same as $\epsilon$,}\\
\eta_n/2&\text{otherwise.}
\end{cases}
\]
We propose the following naive guess of an analogue of \cite[Problem~0.7]{KP} in our current setting:
\begin{equation}\label{eq:guess}
   L_0\stackrel ? =\prod_{n\ge0,e_n>0}\Phi_n^{e_n-1}.
\end{equation}
\end{remark}

\bibliographystyle{amsalpha}
\bibliography{references}

\end{document}